\documentclass[12pt]{amsart}
\usepackage[margin=1in]{geometry}
\usepackage{amsmath}
\usepackage{amsthm}
\usepackage{amsbsy}
\usepackage{amssymb}
\usepackage{hyperref}
\usepackage{amsfonts}
\usepackage{color}

\usepackage{enumitem}
\numberwithin{equation}{section}
\newtheorem{theorem}{Theorem}[section]
\newtheorem{lemma}[theorem]{Lemma}

\newtheorem{corollary}[theorem]{Corollary}

\newtheorem{conjecture}[theorem]{Conjecture}

\theoremstyle{remark}
\newtheorem*{remark}{Remark}

\theoremstyle{remark}

\title[An elementary note on three consecutive powerful numbers]{An elementary note on three consecutive powerful numbers}

\author{Wenjun Ma$^{\sharp}$}
\address{
School of Science, Tianjin University of Technology and Education \\
1310 South Dagu Road, Hexi District, Tianjin 300222, P. R. China}
\email{wjma@tju.edu.cn}

\begin{document}

\begin{abstract}
In the recent work by T.~H. Chan \cite{Chan}, it was shown that there are no three consecutive powerful numbers with the middle term a perfect cube and each of the other two having only a single prime factor raised to an odd power. They get it with Pell equations, elliptic curves, and second-order recurrences. In this paper, with results given by Nagell-Ljunggren, Lebesgue, Ko and some elementary methods, we extend the result to similar cases, whose middle term can be perfect powers for infinitely many powers. With this conclusion, we discuss the solutions to some Diophantine equations.
\end{abstract}

\maketitle

\section{Introduction and Statement of the result}

A positive integer $n$ is called powerful or squareful if $p^2 \mid n$ for all primes $p$ such that $p \mid n$, i.~e., its prime factorization $n=p_1^{a_1}p_2^{a_2}\cdots p_r^{a_r}$ satisfies $a_i\geq2$ for all $1\leq i\leq r$. Similarly, $n$ is $k$-full if $p^k\mid n$ for all prime $p\mid n$.
It is easy to see that any powerful number $n$ can be factored uniquely as $n = a^2b^3$ for some integer $a\geq 1$ and squarefree number $b\geq 1$. Here a number $n$ is called squarefree if $p^2 \nmid n$ for all primes $p$. The following is an initial list of powerful numbers:
$$
1, 4, 8, 9, 16, 25, 27, 32, 36, 49, 64, 72, 81, 100, \ldots.
$$
For more, one can refer OEIS A001694.
Notice that $8$ and $9$ are consecutive powerful numbers. Indeed, $8=2^3$ and $9=3^2$ are perfect powers (that is, powers of exponent higher than one). They are the only perfect powers of natural numbers whose values are consecutive, i.e., the only nontrivial integral solution to the Diophantine equation $x^p-y^q=1$ is given by $3^2-2^3=1$. This is the well-known Catalan's conjecture (or Mih${\rm \breve{a}}$ilescu's theorem), which is raised by E.~C. Catalan in 1844 and proven by Mih${\rm \breve{a}}$ilescu in 2002 \cite{Mihailescu}. There have been many other results on powerful numbers. For example, Erd\"os conjectured that every sufficiently large integer is a sum of at most three powerful numbers, which was proved by Heath-Brown \cite{HeathBrown} \cite{HeathBrown2}.

When considering consecutive powerful numbers, we can find
$$
8, 9, 288, 289, 675, 676, 9800, 9801, \ldots.
$$
For more, one can refer OEIS A060355. In fact, there are infinitely many such pairs.
It has been shown that any solution $x_1,y_1$ of the Pell equation $x^2-dy^2=\pm 1$ with the extra condition that $d|y_1^2$ leads to an infinite family of consecutive powerful numbers. For more details, one can consult \cite{Golomb} and \cite{Walker}.
It's natural to ask if there are three consecutive powerful numbers next. In this topic, Erd\"os \cite{Erdos}, Mollin and Walsh \cite{Mollin} present the following conjecture.
\begin{conjecture}
There are no three consecutive powerful numbers.
\end{conjecture}

This conjecture is very useful. For example, it implies that there are infinitely many non-Wieferich primes (a Wieferich prime is a prime $p$ which is a solution to the congruence equation $2^{p-1}\equiv 1$ mod $p^2$, see \cite{Granville}, \cite{Ribenboim1} and \cite{Vardi}). It appears to be very hard and remains open.
Up to now, it's known that if a triplet of consecutive powerful numbers exists, then its smallest term must be congruent to $7$, $27$, or $35$ modulo $36$.
Conditional on the $abc$-conjecture, it's shown that there are only a finite number of sets of three consecutive powerful numbers. For relevant results, see \cite{Erdos}, \cite{Granville} and \cite{Mollin}.

Recently, Chan \cite{Chan} shows that there are no three consecutive powerful numbers with the middle term a perfect cube and each of the other two having only a single prime factor raised to an odd power, i.e., there are no three consecutive powerful numbers of the form $x^3-1=p^3y^2, x^3, x^3+1=q^3z^2$, with primes $p, q$ and positive integers $x, y, z$.
In this paper, we consider similar cases, whose middle term can be perfect powers for infinitely many powers. Let 
$$
S:=\{{\rm prime}\ p\geq 5\mid p\equiv 5\ {\rm mod}\ 8\}. 
$$
With facts given by Ljunggren, Lebesgue and Ko, we get the following.

\begin{theorem}\label{maintheorem}
There are no three consecutive powerful numbers of the form
$$
x^n-1=q_1^3 y^2, x^n, x^n+1 =q_2^3 z^2,
$$
where $q_1, q_2$ are primes, $x, y, z\in \mathbb{Z}$ and $n\geq 5$ is an integer that only contains prime factors belonging to $S$.
\end{theorem}

Applying this, we can get some results for Diophantine equations. For example, the following corollary follows from it immediately.

\begin{corollary}\label{corollary1}
Suppose $n\geq 5$ is an integer that only contains prime factors belonging to $S$, the Diophantine equation
$$
(2ax)^{2n}-1=q_1^3 q_2^3 y^2
$$
has no solution with integers $a, x, y$ and primes $q_1, q_2$.
\end{corollary}

This paper is organized as follows. In section $2$, we show some results that will be used in the proof. In section $3$ and $4$, the proof of theorem \ref{maintheorem} and corollary \ref{corollary1} is given. In section $5$, we have some remarks. Throughout the paper, all variables are integers. $p, q_1$ and $q_2$ stand for prime numbers.

\section{Preliminaries}

In this section, we give some facts which will be used in the proof.

We present the following lemma on the greatest common divisor of $x\pm 1$ and $\frac{x^p\pm 1}{x\pm 1}$ first.

\begin{lemma}\label{gcd}
Suppose $|x|>1$ be an integer and $p$ be an odd prime, then
\begin{eqnarray*}
&&{\rm gcd}\left(x-1, \frac{x^p-1}{x-1}\right)=\left\{\begin{array}{cc}
                                                 p & {\rm if}\ x\equiv 1 \ {\rm mod}\ p,  \\
                                                 1 & {\rm otherwise},
                                               \end{array}
                                               \right.   \\
&&{\rm gcd}\left(x+1, \frac{x^p+1}{x+1}\right)=\left\{\begin{array}{cc}
                                                 p & {\rm if}\ x\equiv -1 \ {\rm mod}\ p,  \\
                                                 1 & {\rm otherwise}.
                                               \end{array}
                                               \right.
\end{eqnarray*}
\end{lemma}
\begin{proof}
With the fact that ${\rm gcd}\left(a, b\right)={\rm gcd}\left(a, b-a\right)$ and ${\rm gcd}\left(a, b\right)={\rm gcd}\left(a, -b\right)$, we have
\begin{eqnarray*}
{\rm gcd}\left(x-1, \frac{x^n-1}{x-1}\right)
&=& {\rm gcd}\left(x-1, x^{n-1}+x^{n-2}+\cdots+x+1\right) \\
&=& {\rm gcd}\left(x-1, (x-1+1)^{n-1}+(x-1+1)^{n-2}+\cdots+(x-1+1)+1\right) \\
&=& {\rm gcd}\left(x-1, n\right)
\end{eqnarray*}
for all positive integer $n$ and
\begin{eqnarray*}
{\rm gcd}\left(x+1, \frac{x^{2k+1}+1}{x+1}\right)
&=& {\rm gcd}\left(x+1, x^{2k+1}-x^{2k}+\cdots+x-1\right) \\
&=& {\rm gcd}\left(x+1, (x+1-1)^{{2k+1}}-(x+1-1)^{2k}+\cdots+(x+1-1)-1\right) \\
&=& {\rm gcd}\left(x+1, -(2k+1)\right) \\
&=& {\rm gcd}\left(x+1, 2k+1\right)
\end{eqnarray*}
for all non-negative integer $k$.
It follows that
\begin{eqnarray*}
&&{\rm gcd}\left(x-1, \frac{x^p-1}{x-1}\right) = {\rm gcd}\left(x-1, p\right)
=\left\{\begin{array}{cc}
                                                 p & {\rm if}\ x\equiv 1 \ {\rm mod}\ p,  \\
                                                 1 & {\rm otherwise},
                                               \end{array}
                                               \right. \\
&&{\rm gcd}\left(x+1, \frac{x^p+1}{x+1}\right) = {\rm gcd}\left(x+1, p\right)
= \left\{\begin{array}{cc}
                                                 p & {\rm if}\ x\equiv -1 \ {\rm mod}\ p,  \\
                                                 1 & {\rm otherwise}.
                                               \end{array}
                                               \right.
\end{eqnarray*}
for odd prime $p$.
\end{proof}

\begin{remark}
In fact, as
$$
\frac{x^n-y^n}{x-y}=\frac{(x-y+y)^n-y^n}{x-y}=k(x-y)+ny^{n-1}
$$
for some integer $k$, we have
$$
{\rm gcd}\left(x-y, \frac{x^n-y^n}{x-y}\right)={\rm gcd}(x-y, n)
$$
for coprime integers $x, y$, which also leads to the conclusion.
\end{remark}

The following work by Lebesgue \cite{Lebesgue} is needed, in order to deal with equation $x^m-y^2=1$. See also Cassels \cite{Cassels} and Tang \cite{Tang}, who reproduced Lebesgues's proof.

\begin{lemma}\label{Lebesgue}
With $m\geq 2$, the equation $x^m-y^2=1$ has no solution in positive integers.
\end{lemma}

\begin{proof}
If $m$ is even, $x^m-y^2=1$ turns to be $(x^{m/2}-y)(x^{m/2}+y)=1$, which means $x^{m/2}-y=x^{m/2}+y=1$. The result turns out immediately. Hence assume that $m$ is odd. Let $x, y$ be positive integers such that $x^m=y^2+1$. If $y$ is odd, we have $x^m \equiv 2$ mod $4$, which is impossible. Therefore $y$ is even and $x$ is odd.

Consider the quadratic field $\mathbb{Q}(i)$, where $i=\sqrt{-1}$. Then gcd$(y-i, y+i)$ is a unit of $\mathbb{Q}(i)$. With $x^m=(y+i)(y-i)$, it follows that the factor $y+i$ is a m-th power in $\mathbb{Q}(i)$ up to a unit, that is, $y+i=(u+iv)^m i^s, 0\leq s<3, u, v\in\mathbb{Z}$. Hence $y-i=(u-iv)^m (-i)^s$. Therefore $x^m=(u^2+v^2)^m$, which leads to $x=u^2+v^2$. Then we have $u$ or $v$ is even, as $x$ is odd.
As
$$
2i= y+i-(y-i)=((u+iv)^m-(u-iv)^m(-1)^s)i^s,
$$
if $s$ is even, we get
$$
1 = (-1)^r \left( m u^{m-1} v - \binom{m}{3} u^{m-3} v^3 + \cdots \pm v^m \right).
$$
Thus $v\mid 1$, hence $v=\pm 1$ is odd.
If $s$ is odd, we have
$$
1 = (-1)^r \left( u^m - \binom{m}{2} u^{m-2} v^2 + \ldots \pm m u v^{m-1} \right),
$$
which tells $u=\pm 1$ is odd.

Let $w$ be the even one of $u$ and $v$. Then
$$
1 - \binom{m}{2} w^2 + \binom{m}{4} w^4 - \ldots \pm m w^{m-1} = \pm 1.
$$
If the sign on the right of the above equation is $-$, we get $w^2\mid 2$, which contradicts with $w$ is even. If the sign is $+$, we obtain
$$
\binom{m}{2} - \binom{m}{4} w^2 + \ldots \pm m w^{m-3} = 0.
$$
As $w$ is even, we know $\binom{m}{2}$ is even. Let $2^r \| \binom{m}{2}$ and $2^l\| \binom{m}{2k} w^{2k-2}, k\geq 2$. With
$$
\binom{m}{2k} w^{2k-2} = \binom{m}{2} \binom{m-2}{2k-2} \times \frac{2}{2k(2k-1)} w^{2k-2},
$$
we get $l\geq r+1$. Therefore the case with sign $+$ is impossible. This concludes the proof of the lemma.

\end{proof}

We also need to deal with equation $x^m-y^2=-1$. Hence we present the result shown by Chao Ko \cite{Ko1} \cite{Ko2}. One can also refer to Chein \cite{Chein}.

\begin{lemma}\label{Ko}
With $n>3$, the equation $x^n-y^2=-1$ has no solution in positive integers.
\end{lemma}

\begin{proof}
If $n$ is even, the conclusion is trivial, as there is no positive integer solution for $x^2-y^2=-1$. If $n=3$, the only solution is $x=2, y=3$ due to Mih${\rm \breve{a}}$ilescu's beautiful work on Catalan's conjecture. Therefore, it suffices to show that if $p>3$ is a prime, the equation $x^p-y^2=-1$ has no solution in positive integers.
Suppose that $x, y$ are positive integers such that $x^p = y^2 - 1$.
If $y$ is even, gcd$(y+1, y-1)=$gcd$(y+1, 2)=1$, we get $y+1=c^p, y-1=d^p$ for some positive integers $c, d$. It follows that $c^p-d^p=2$, which is impossible. Hence $y$ has to be odd.

If $p\nmid y$, with lemma \ref{gcd}, we know
$$
{\rm gcd}\left(x+1, \frac{x^p+1}{x+1}\right)=1.
$$
So there exists integers $u>1, v>0$ such that
$$
\begin{cases}
x+1=u^2, \\
\frac{x^p+1}{x+1}=v^2.
\end{cases}
$$
Hence $y^2-(u^2-1)((u^2-1)^{(p-1)/2})^2=1$, which means $(y, (u^2-1)^{(p-1)/2})$ is a solution of Pell equation $X^2-(u^2-1)Y^2=1$. The fundamental solution is $(u, 1)$, because $u^2-(u^2-1)=1$. Let $(x_m, y_m), m\geq 1$ be the solutions of $X^2-(u^2-1)Y^2=1$ in positive integers. We have $$
x_m+y_m \sqrt{u^2-1} = (u + 1\cdot\sqrt{u^2-1})^m.
$$
Expanding the right of the above equation, not hard to see we couldn't find $m\geq 1$ such that $y_m = (u^2-1)^{(p-1)/2}$, which contradicts with $(y, (u^2-1)^{(p-1)/2})$ is a solution. Therefore $p\mid y$.

Since $y$ has to be odd, we have gcd$(y+1, y-1)=2$. It follows that either
$$
(I)\ \begin{cases} y + 1 = 2a^p, \\ y - 1 = 2^{p-1}b^p, \end{cases} \ \ {\rm or}\ \ (II)\ \begin{cases} y + 1 = 2^{p-1}b^p, \\ y - 1 = 2a^p, \end{cases}.
$$
with $a$ odd, $a, b$ coprime positive integers and $y=2ab$.
Hence
$$
a^p-2^{p-2}b^p=\pm 1
$$
according to the case. Then
$$
(a^2)^p \mp (2b)^p = (a^p \mp 2)^2 = \left( \frac{y \mp 3}{2} \right)^2.
$$
Since $p\mid y$ and $p>3$, we get $p\nmid \frac{y\mp 3}{2}$. Hence $p\nmid {\rm gcd}\left(a^{2} \mp 2b, \frac{(a^{2})^{p} \mp (2b)^{p}}{a^{2} \mp 2b}\right)$.
It follows that
$$
{\rm gcd}\left(a^{2} \mp 2b, \frac{(a^{2})^{p} \mp (2b)^{p}}{a^{2} \mp 2b}\right) = {\rm gcd}\left(a^{2} \mp 2b, p\right)=1.
$$
Therefore $a^2\mp 2b=h^2$, where $h$ divides $(y\mp 3)/2$. Note that $h$ is odd, with $a^2-h^2=\pm 2b$, we know $b$ is even. Then
$$
(ha)^2+b^2=(a^2\mp b)^2.
$$
The positive integers $ha, b, a^{2} \mp b$ are relatively prime. They constitute a primitive solution of $X^2+Y^2=Z^2$. Hence there exist integers $c, d\geq 1$ such that
$$
\begin{cases}
ha = c^{2} - d^{2}, \\
b = 2cd, \\
a^{2} \mp b = c^{2} + d^{2}.
\end{cases}
$$
Then $(c \pm d)^{2} = (a^{2} \mp b) \pm b = a^{2}$. In the first case
$$
b - a = 2cd - (c + d) = (c - 1)(d - 1) + (cd - 1) > 0,$$
that is, $a < b$. However $a^{q} = 2^{q-2}b^{q} + 1 > b^{q}$, so $a > b$, a contradiction. In the second case,
$$b - a = 2cd - (c - d) = c(2d - 1) + d > 0.$$
So $a < b$. However, $a^{q} = 2^{q-2}b^{q} - 1 > b^{q}$, hence $a > b$, which is also impossible.
\end{proof}

\begin{remark}
With lemma \ref{Lebesgue} and \ref{Ko}, it is not hard to see, the only integer solution of $x^m-y^2=1$ is $(x, y)=(1, 0)$ with $m$ being an odd prime, and the the only integer solution of $x^n-y^2=-1$ is $(x, y)=(0, \pm 1)$ with $n>3$ being an odd prime.
\end{remark}

By the end of this section, we introduce the celebrated work by Nagell and Ljunggren.

\begin{lemma}\label{Ljunggren}
If $x, n$ are integers such that $|x|>1$ and $n>2$, $\frac{x^n-1}{x-1}$ can't be a perfect square number except for $n=4, x=7$ or $n=5, x=3$.
\end{lemma}

The proof of it first appears in Ljunggren's paper \cite{Ljunggren}. For English version, one may refer \cite{Ribenboim2}(see A8.1). Here is a sketch of the proof below.

\begin{proof}
When $n=3$, it's obvious that $x^2+x+1$ isn't a perfect square number, as $x^2<x^2+x+1<(x+1)^2$, $(x-1)^2<x^2-x+1<x^2$ for all positive integer $x>1$.
We assume $n>3$ in the following and show the proof in three parts.

\textbf{Case $1$}: $4\mid n$.

Write $n=2^a m$, with $m$ odd and $a\geq 2$. Let $z=x^m$. Suppose
\begin{eqnarray*}
\frac{x^n-1}{x-1} &=& \frac{z^{2^a}-1}{z-1} \cdot \frac{x^m-1}{x-1} \\
&=& \frac{x^m-1}{x-1}\cdot (z+1)(z^2+1)\cdots (z^{2^{a-1}}+1)
\end{eqnarray*}
is a perfect square.
We have the facts that
\begin{eqnarray*}
&&{\rm gcd}\left(z^{2^j}+1, z^{2^i}+1\right)=\left\{\begin{array}{cc}
                                                 1 & {\rm if}\ z \ {\rm is\  even},  \\
                                                 2 & {\rm if}\ z \ {\rm is\   odd}, \
                                               \end{array}
                                               \right.   \\
&&{\rm gcd}\left(\frac{x^m-1}{x-1}, x^{2^i m}+1\right) = 1,
\end{eqnarray*}
where $j, i$ are distinct non-negative integers.
Thus, if $z$ is even, each of $\frac{x^m-1}{x-1}, (z+1), (z^2+1), \ldots,  (z^{2^{a-1}}+1)$ is a square. In particular, $z^2+1$ is a square, which is impossible. Hence $z$ is odd.

If $a=2$, we have
$$
\begin{cases}
z+1 = 2 e^2 ,  \\
z^2+1 = 2 f^2,
\end{cases}
$$
where $e, f$ are positive integers.
It follows that $x^m=z= \pm 1$ or $7$. With the condition that $|x|>1$, we get $m=1, n=4, x=7$.

If $a>2$, by similar argument, we have
$$
\begin{cases}
z^2+1 = 2 e^2 ,  \\
z^4+1 = 2 f^2
\end{cases}
$$
for some positive integers $e, f$. Therefore $z^2=1$ or $7$, which means $x=\pm 1$, contradicting with $|x|>1$.

\textbf{Case $2$}: $n>3$ is even.

In virtue of Case $1$, we may assume $n=2m$, where $m\geq 2$ is odd. Suppose
\begin{eqnarray*}
\frac{x^n-1}{x-1} &=& \frac{x^m-1}{x-1}(x^m+1).
\end{eqnarray*}
is a perfect square. By the fact that
$$
{\rm gcd}\left(\frac{x^m-1}{x-1}, x^m+1\right)=1,
$$
there exist integers $s, t$ such that gcd$(s, t)=1$ and
$$
\begin{cases}
\frac{x^m-1}{x-1} = s^2 ,  \\
x^m+1=t^2.
\end{cases}
$$
According to lemma \ref{Ko}, this is impossible.

\textbf{Case $3$}: $n>3$ is odd.

Suppose $\frac{x^n-1}{x-1} = y^2$ for some integer $y$.
Assume $x>1$, we have
$$
(x^n-1)(x-1)=((x-1)y)^2,
$$
which leads to
$$
((x-1)y)^2-x(x-1)(x^{(n-1)/2})^2=-(x-1).
$$
Write $x-1=t^2 s$, where $s$ is a square-free integer and $t$ is a positive integer. Hence
$$
(tsy)^2-xs(x^{(n-1)/2})^2=-s.
$$
Knowing that gcd$(x, s)=1$ and $s$ is square-free, it's easy to get that $xs$ is not a square.
Consider the equation
$$
X^2-DY^2=C,
$$
where $D=xs, C=-s$. This equation has fundamental solution $(ts, 1)$, as
$$
t^2((ts)^2-(xs)\cdot 1^2)=(t^2s)^2-xt^2s =(x-1)^2-x(x-1)=-(x-1)=-t^2s.
$$
It follows that all the solutions in positive integers are $(x_m, y_m)$, where $m\geq 1$ is odd, and
$$
s^{(m-1)/2}(x_m+y_m\sqrt{xs})=(ts+\sqrt{xs})^m
$$
In particular, there exists odd $m\geq 1$ such that
$$
tsy=x_m,  \ \  x^{(n-1)/2}=y_m.
$$
Note that if a prime $q$ divides $y_m$, then $q$ divides $D=xs$.
We have $m=1$ or $3$. If $m=1$, $x^{(n-1)/2}=y_1=1$, which contradicts with $x>1$. If $m=3$, we get
$$
x^{(n-1)/2}=y_3=\frac{3t^2s^2+xs}{s}=4x-3,
$$
which means $3=x(4-x^{(n-3)/2})$. Therefore $x=1$(excluded) or $x=3, n=5$.

Assume $x<-1$, let $z=-x>1$. Then
$$
\frac{z^n+1}{z+1}=y^2,
$$
which leads to
$$
((z+1)y)^2-z(z+1)(z^{(n-1)/2})^2=z+1.
$$
Let $z+1=t^2s$, where $s$ is square-free. We have
$$
(tsy)^2-zs(z^{(n-1)/2})^2=s.
$$
With similar arguments as case $x>1$, we get that there's no integer solution to it.

\end{proof}

\section{Proof of Theorem \ref{maintheorem}}

In this section, the proof of theorem \ref{maintheorem} is given, with facts shown in section $2$.
Not hard to see, theorem \ref{maintheorem} is equal to say
there are no three consecutive powerful numbers of the form
$$
x^p-1=q_1^3 y^2, x^p, x^p+1 =q_2^3 z^2,
$$
where $q_1, q_2$ are primes, $x, y, z\in \mathbb{Z}$ and $p \equiv 5$ mod $8$ is an odd prime.

Suppose that there exists $q_1, q_2, x, y, z$ and $p$ defined as above, such that $x^p-1=q_1^3 y^2, x^p, x^p+1 =q_2^3 z^2$ are all powerful numbers. Easy to see when $|x|\leq 1$, the triple won't be consecutive powerful numbers. Thus $|x|>1$. We will derive contradictions in three scenarios.

\textbf{Case $1$}: $x\not\equiv \pm 1$ mod $p$.

Knowing that $p$ is an odd prime, with lemma \ref{gcd}, we have
$$
{\rm gcd}\left(x-1, \frac{x^p-1}{x-1}\right)={\rm gcd}\left(x+1, \frac{x^p+1}{x+1}\right)=1.
$$
If $q_1\mid x-1$, we have $q_1^3\mid x-1$ and $\frac{x^p-1}{x-1}$ is a perfect square. Together with the fact that when $x=3, p=5$, $x^p-1=242$ isn't a powerful number, the later contradicts with lemma \ref{Ljunggren}. Hence we have $q_1\mid \frac{x^p-1}{x-1}$ and $x-1$ is a perfect square.

Applying similar argument, we have $q_2\mid \frac{x^p+1}{x+1}$ and $x+1$ ia a perfect square. It follows that both $x-1$ and $x+1$ are perfect squares, which is impossible, as we always have $(n+1)^2-n^2=2n+1>2$ for integer $n>1$.

\textbf{Case $2$}: $x\equiv 1$ mod $p$.

With lemma \ref{gcd}, we have
$$
{\rm gcd}\left(x-1, \frac{x^p-1}{x-1}\right)=p,
$$
$$
{\rm gcd}\left(x+1, \frac{x^p+1}{x+1}\right)=1.
$$
The former implies $x-1=pv^2$, $pq_1v^2$ or $(pv)^2$ and the later implies $q_2\mid \frac{x^p+1}{x+1}, x+1=t^2$ for some integers $v, t$.
Hence we have $t^2 \equiv 2$ mod $p$, i.e., $2$ is a quadratic residue mod $p$, which is equivalent to $p \equiv \pm 1$ mod $8$. This contradicts the assumption $p \equiv 5$ mod $8$.

\textbf{Case $3$}: $x\equiv -1$ mod $p$.

By similar argument as case $2$, we get $s^2\equiv -2$ mod $p$, i.e., $-2$ is a quadratic residue mod $p$, which is equivalent to $p \equiv 1$ or $3$ mod $8$. It also yields a contradiction with $p \equiv 5$ mod $8$. This concludes the proof of the theorem.

\section{Proof of Corollary \ref{corollary1}}

Applying theorem \ref{maintheorem}, we provide the proof of corollary \ref{corollary1} is this section.

Suppose the Diophantine equation
$$
(2ax)^{2n}-1=((2ax)^n-1)((2ax)^n+1)=q_1^3 q_2^3 y^2
$$
has a solution with some integers $a, x, y$, primes $q_1, q_2$ and integer $n\geq 5$ that only contains prime factors belonging to $S$. We have
$$
{\rm gcd}((2ax)^n-1, (2ax)^n+1)={\rm gcd}((2ax)^n-1, 2)=1.
$$
Hence we must have
$$
q_1 q_2 \mid (2ax)^n-1, q_1 q_2 \nmid (2ax)^n+1,
$$
$$
q_1 q_2 \nmid (2ax)^n-1, q_1 q_2 \mid (2ax)^n+1,
$$
or $q_1$ divides exactly one of $(2ax)^n-1, (2ax)^n+1$ and $q_2$ divides the other one.

If $q_1 q_2 \mid (2ax)^n-1, q_1 q_2 \nmid (2ax)^n+1$, we have $(2ax)^n+1 = u^2$ for some integer $u$. Substituting $2ax=X, u=Y$, this is just $X^n-Y^2=-1$, which has no integer solution except for $(X, Y)=(0, \pm 1)$ due to lemma \ref{Ko}. Thus $a=x=0$. However, $0^{2n}-1=-1$ isn't of the form $q_1^3 q_2^3 y^2$.

If $q_1 q_2 \nmid (2ax)^n-1, q_1 q_2 \mid (2ax)^n+1$, we have $(2ax)^n-1 = v^2$ for some integer $v$. Substituting $2ax=X, u=Y$, this becomes $X^n-Y^2=1$, which has no integer solution except for $(X, Y)=(1, 0)$ due to lemma \ref{Lebesgue}. Thus $2ax=1$, which has no integer solution.

The only case left is $q_1$ divides exactly one of $(2ax)^n-1, (2ax)^n+1$ and $q_2$ divides the other one. Without loss of generality, suppose $(2ax)^n-1=q_1^3 y_1^2$ and $(2ax)^n+1=q_2^3 y_2^2$ for some integers $y_1, y_2$. This contradicts Theorem \ref{maintheorem}. The corollary follows.

\section{Remarks}

\begin{remark}
We get the result with elementary tools. If apply some other methods, hopefully the result could be extended to some other prime $p$.

The reason we need to restrict $p$ to the condition that $p\equiv 5$ mod $8$ is to "kill" case $2$ and $3$ with quadratic residue in the proof of theorem \ref{maintheorem}. There are chances we use other tools, such as properties of Diophantine equations instead of quadratic residue, to get better result.

For example, in case $2$ of the proof of theorem \ref{maintheorem}, we can have $p\neq q_1$ first, which leads to the non-existence of $x-1=pq_1 v^2$, and draw out $x-1=(pv)^2$ by the fact it is a perfect square. Hence the only remaining is $x-1=pv^2$, which yields $t^2-pv^2=2$. Applying Diophantine equation methods, there are chances that we can go on with it. Similar phenomenon appears in case $3$. Of course there are some new points to deal with when applying this.
\end{remark}


\begin{remark}
Our result may extend to other cases, such as 3-full numbers. When $q=3, n\not\equiv 5$ mod $6$, $x, y, n>1$, Nagell-Ljunggren equation $\frac{x^n-1}{x-1}=y^q$ is known to have no integer solution except for $(x, y, n, q)=(18, 7, 3, 3)$(see \cite{Ljunggren}). Applying this, we will show that $x^p-1=q_1^{\alpha_1}y^3, x^p, x^p+1=q_2^{\alpha_2}z^3$ can't be consecutive 3-full numbers when $p\not\equiv 5$ mod $6$ and neither $2$ nor $-2$ is cubical residue mod $p$.
\end{remark}

\end{document}